\documentclass{amsart}
\usepackage{amssymb}
\usepackage{amsmath}
\usepackage{amsthm}
\usepackage{mathrsfs}
\usepackage{verbatim}
\usepackage{xcolor}
\usepackage{graphicx}
\newtheorem{theorem}{Theorem}[section]

\newtheorem{lemma}[theorem]{Lemma}
\newtheorem{corollary}[theorem]{Corollary}
\newtheorem*{corollary*}{Corollary}
\theoremstyle{definition}
\newtheorem{definition}[theorem]{Definition}

\theoremstyle{remark}
\newtheorem{remark}[theorem]{Remark}
\numberwithin{equation}{section}
\title[Comparison principle for the CMAF]{A general comparison principle for the pluripotential complex Monge-Amp\`ere flow}
\author{Bowoo Kang}
\address{Department of Mathematical Sciences, KAIST, 291 Daehak-ro, Yuseong-gu, Daejeon
34141, South Korea}
\email{bou704@kaist.ac.kr}
\begin{document}
\begin{abstract}
    We prove a comparison principle for the pluripotential complex Monge-Amp\`ere flows for the right-hand side of the form $dt \wedge d\mu$ where $d\mu$ is dominated by a Monge-Amp\`ere measure of a bounded plurisubharmonic function. As a consequence, we obtain the uniqueness of the weak solution to the pluripotential Cauchy-Dirichlet problem. We also study the long-term behavior of the solution under some assumption.
\end{abstract}
\maketitle
\section{Introduction}
    \indent In this paper, we aim at extending parabolic pluripotential theory introduced by Guedj, Lu and Zeriahi in \cite{GLZ21-2}. Pluripotential theory for the complex Monge-Amp\`ere equation has been well developed, starting from the works by Bedford and Taylor \cite{BT76, BT82}. Guedj, Lu and Zeriahi \cite{GLZ20, GLZ21-2} successfully developed the parabolic pluripotential theory and studied the weak solutions for the pluripotential complex Monge-Amp\`ere flows. They proved the existence and uniqueness of the solution to the Cauchy-Dirichlet problem when the measure on the right-hand side is given a positive $L^p$ density function for some $p > 1$. We considerably extend their results by studying the Cauchy-Dirichlet problem for general measures.\\
    \indent We recall the setup of \cite{GLZ21-2}, which will be used throughout this paper. Let $\Omega \Subset \mathbb{C}^n$ be a bounded strictly pseudoconvex domain. From now on, we fix $0 < T < +\infty$ and denote $\Omega_T = (0, T) \times \Omega$. We say that a function $u : \Omega_T \rightarrow [-\infty, +\infty)$ is locally uniformly Lipschitz in $(0, T)$ if for any compact $J \Subset (0, T)$, there exists $ \kappa_J > 0$ satisfying
    \begin{align*}
        u(t, z) \leq u(s, z) + \kappa_J\lvert t-s\rvert \text{ for all $t, s \in J$ and for all $z \in \Omega$.}
    \end{align*}
    Next, we say that a function $u : \Omega_T \rightarrow \mathbb{R}$ is locally uniformly semi-concave in $(0, T)$ if for any compact $J \Subset (0, T)$, there exists $C_J > 0$ such that
    \begin{align*}
        t \mapsto u(t, z)-C_Jt^2 \text{ is concave in $J$ for all $z \in \Omega$.}
    \end{align*}
    The family of \textit{parabolic potentials} $\mathcal{P}(\Omega_T)$ is the set of functions $u$ such that
    \begin{itemize}
        \item for any $t \in (0, T)$, $u(t, \cdot) \in PSH(\Omega)$,
        \item $u$ is locally uniformly Lipschitz in $(0, T)$. 
    \end{itemize}
    The \textit{Cauchy-Dirichlet boundary data} $h$ is a function defined on $\partial_0\Omega_T := ([0, T) \times \partial \Omega) \cup (\{0\} \times \Omega)$ satisfying
    \begin{itemize}
        \item the restriction of $h$ on $[0, T) \times \partial \Omega$ is continuous,
        \item the family $\{h(\cdot, z)~:~ z \in \partial\Omega\}$ is locally uniformly Lipschitz in $(0, T)$,
        \item for any $\zeta \in \partial \Omega$,
        \begin{align*}
            h_0 := h(0, \cdot) \in PSH(\Omega) \cap L^{\infty}(\Omega) \text{ and } \lim_{\Omega \ni z \rightarrow \zeta}h(0, z) = h(0, \zeta).
        \end{align*}
    \end{itemize}
    Let $F(t, z, r) : [0, T) \times \Omega \times \mathbb{R} \rightarrow \mathbb{R}$ be a continuous function which is increasing in $r$, bounded in $[0, T) \times \Omega \times J$ for each $J \Subset \mathbb{R}$ and $(t, r) \mapsto F(t, \cdot, r)$ is locally uniformly Lipschitz and semi-convex. Let $\mu$ be a non-negative Borel measure on $\Omega$ such that $PSH(\Omega) \subset L^1_{loc}(\Omega, \mu)$.  \\
    \indent We consider the following complex Monge-Amp\`ere flows
    \begin{align}\label{degenerate cmaf}
        dt \wedge (dd^cu)^n = e^{\partial_tu+F(t, z, u)}dt \wedge d\mu \text{ in } \Omega_T,
    \end{align}
    where $u \in \mathcal{P}(\Omega_T) \cap L^{\infty}(\Omega_T)$ is an unknown function. Here, $dt \wedge (dd^cu)^n$ is understood in the pluripotential sense as defined in \cite[Definition 2.2]{GLZ21-2}. The Cauchy-Dirichlet problem for the (\ref{degenerate cmaf}) with the Cauchy-Dirichlet boundary data $h$ is finding $u \in \mathcal{P}(\Omega_T) \cap L^{\infty}(\Omega_T)$ such that $u$ solves (\ref{degenerate cmaf}) and
    \begin{align}\label{Cauchy-Dirichlet boundary condition}
        \begin{cases}
            &\lim_{(t, z) \rightarrow (\tau, \zeta)}u(t, z) = h(\tau, \zeta) \text{ for all } (\tau, \zeta) \in [0, T) \times \partial \Omega, \\
        &\lim_{t \rightarrow 0+}u(t, \cdot) = h_0 \text{ in } L^1(\Omega, d\mu).
        \end{cases}
    \end{align}
    Let $dV$ be a Lebesgue measure on $\mathbb{C}^n$. Guedj, Lu and Zeriahi \cite{GLZ21-2} proved the existence and uniqueness of the solution to the Cauchy-Dirichlet problem with a measure $\mu$ and a boundary data $h$ satisfying
    \begin{align*}
        d\mu = gdV \text{ for some } g \in L^p(\Omega), ~p > 1 \text{ with } g > 0 \text{ a.e.},
    \end{align*}
    and there exists $C > 0$ such that 
    \begin{align}\label{additional assumption}
        t\lvert \partial_th(t, z)\rvert \leq C \text{ for all } (t, z) \in (0, T) \times \partial \Omega,
    \end{align}
    and
    \begin{align}\label{additional assumption 2}
        t^2\partial_{tt}h(t, z) \leq C \text{ in the sense of distributions in } (0, T) \text{ for all } z \in \Omega.
    \end{align}
    We focus on their uniqueness result. The authors first proved the following comparison principle and then used this to prove the uniqueness of the solution which is locally uniformly semi-concave in $(0, T)$.
    \begin{theorem}{\cite[Theorem 6.6]{GLZ21-2}}\label{previous comparison principle by GLZ}
    Let $d\mu = gdV$ for some $g \in L^p(\Omega)$ with $p > 1$ and $g > 0$ a.e. with respect to the Lebesgue measure on $\Omega$. Let $u, v \in \mathcal{P}(\Omega_T) \cap L^{\infty}(\Omega_T)$ with the Cauchy-Dirichlet boundary data $h_{1}$, $h_{2}$ as in (\ref{Cauchy-Dirichlet boundary condition}). Assume that
    \begin{enumerate}
        \item [(a)] $v$ is locally uniformly semi-concave in $t \in (0, T)$,
        \item [(b)] $dt \wedge (dd^cu)^n \geq e^{\partial_tu + F(t, z, u)}dt \wedge d\mu$ in $\Omega_T$,
        \item [(c)] $dt \wedge (dd^cv)^n \leq e^{\partial_tv + F(t, z, v)}dt \wedge d\mu$ in $\Omega_T$,
        \item [(d)] $h_1$ satisfies (\ref{additional assumption}) and (\ref{additional assumption 2}).  
    \end{enumerate}
    If $h_1 \leq h_2$, then $u \leq v$.
\end{theorem}
    Our main result generalizes this comparison principle for a measure $\mu$ such that there exists $\varphi \in PSH(\Omega) \cap L^{\infty}(\Omega)$ satisfying
    \begin{align}\label{MA}
        \begin{cases}
            &(dd^c\varphi)^n \geq d\mu \text{ in } \Omega, \\
            &\lim_{z \rightarrow \partial \Omega}\varphi(z) = 0.
        \end{cases}
    \end{align}
    Thanks to the Ko{\l}odziej's subsolution theorem, it is equivalent to the existence of $\psi \in PSH(\Omega) \cap L^{\infty}(\Omega)$ satisfying
    \begin{align}\label{MA2}
        \begin{cases}
            &(dd^c\psi)^n = d\mu \text{ in } \Omega, \\
            &\lim_{z \rightarrow \partial \Omega}\psi(z) = 0.
        \end{cases}
    \end{align}
    From now, we fix such $\mu$, $\varphi$ and $\psi$ throughout this paper.
\begin{theorem}\label{main result 1}
        Let $u, v \in \mathcal{P}(\Omega_T) \cap L^{\infty}(\Omega_T)$ with the Cauchy-Dirichlet boundary data $h_1$, $h_2$ as in (\ref{Cauchy-Dirichlet boundary condition}). Assume that
    \begin{enumerate}
        \item [(a)] $u$ and $v$ are locally uniformly semi-concave in $(0, T)$,
        \item [(b)] $dt \wedge (dd^cu)^n \geq e^{\partial_tu+F(t, z, u)}dt \wedge d\mu$ in $\Omega_T$,
        \item [(c)] $dt \wedge (dd^cv)^n \leq e^{\partial_tv+F(t, z, v)}dt \wedge d\mu$ in $\Omega_T$,
        \item [(d)] $h_1$ satisfies (\ref{additional assumption}).
    \end{enumerate}
    If $h_1 \leq h_2$, then $u \leq v$.
\end{theorem}
\begin{remark}\label{remark 1.3}
    Note that the conditions (a) and (d) in Theorem \ref{previous comparison principle by GLZ} are different from the ones of Theorem \ref{main result 1}. In \cite{GLZ21-2}, instead of comparing a subsolution and a supersolution, the authors compared an upper envelope of a subsolution and a supersolution. They used the assumption (\ref{additional assumption 2}) in \cite[Theorem 4.8]{GLZ21-2} to show that the upper envelope is locally uniformly semi-concave, which was used to prove the (joint) continuity of the envelope \cite[Theorem 6.4, Theorem 6.5]{GLZ21-2}. Also, due to their assumption on the measure, they could obtain the continuity of a supersolution. In this paper, we do not have the (joint) continuity of the upper envelope and a supersolution in general, so we need a different approach. Instead, we directly compare two locally uniformly semi-concave subsolution and supersolution without the assumption (\ref{additional assumption 2}).
\end{remark}
The existence of a solution to the Cauchy-Dirichlet problem under assumptions (\ref{additional assumption}), (\ref{additional assumption 2}), and (\ref{MA}) is proved in our previous work \cite{Ka25}. By using Theorem \ref{main result 1}, we now obtain the uniqueness of the solution to the Cauchy-Dirichlet problem, which is locally uniformly semi-concave in $(0, T)$.
\begin{corollary}\label{unique existence}
    Let $h$ be a Cauchy-Dirichlet boundary data satisfying (\ref{additional assumption}) and (\ref{additional assumption 2}). Then there exists a unique $u \in \mathcal{P}(\Omega_T) \cap L^{\infty}(\Omega_T)$ which is locally uniformly semi-concave in $(0, T)$ and satisfy
    \begin{align*}
        \begin{cases}
            &dt \wedge (dd^cu)^n = e^{\partial_tu+F(t, z, u)}dt \wedge d\mu \text{ in } \Omega_T, \\
            &\lim_{(t, z) \rightarrow (\tau, \zeta)}u(t, z) = h(\tau, \zeta) \text{ for all } (\tau, \zeta) \in [0, T)\times \partial \Omega, \\
            &\lim_{t \rightarrow 0+}u(t, \cdot) = h(0, \cdot) \text{ in } L^1(d\mu).
        \end{cases}
    \end{align*}
\end{corollary}
Even though we fixed $0 < T < +\infty$, the result of Corollary \ref{unique existence} covers the case when $T = +\infty$, which will be explained in Remark \ref{infinite time case}. As an application of Theorem \ref{main result 1}, we also show the following long-term behavior of the solution under some assumption on $F$.
\begin{theorem}
    Let us denote $T = +\infty$. Let $u \in \mathcal{P}(\Omega_T) \cap L^{\infty}(\Omega_T)$ be the unique function which is locally uniformly semi-concave in $(0, T)$ and satisfy
    \begin{align*}
        \begin{cases}
            & dt \wedge (dd^cu)^n = e^{\partial_tu+F(z, u)}dt \wedge d\mu \text{ in } \Omega_T, \\
            &\lim_{(t, z) \rightarrow (\tau, \zeta)}u(t, z) = 0 \text{ for all } (\tau, \zeta) \in [0, T) \times \partial \Omega, \\
            &\lim_{t \rightarrow 0+}u(t, \cdot) = \psi(\cdot) \text{ in } L^1(d\mu),
        \end{cases}
    \end{align*}
    where $\psi$ is given in (\ref{MA2}). If there exists a constant $L_F > 0$ such that 
    \begin{align*}
        \lvert F(z, r_1) - F(z, r_2)\rvert \geq L_F\lvert r_1-r_2\rvert
    \end{align*}
    for all $z \in \Omega$ and $r_1, r_2 \in J$, then $u(t, \cdot) \rightarrow \psi_{\infty}$ uniformly as $t \rightarrow +\infty$ for $\psi_{\infty} \in PSH(\Omega) \cap L^{\infty}(\Omega)$ satisfying
    \begin{align*}
        \begin{cases}
            &(dd^c\psi_{\infty})^n = e^{F(z, \psi_{\infty})}d\mu \text{ in } \Omega,\\
            &\lim_{z \rightarrow \partial \Omega}\psi_{\infty}(z) = 0.
        \end{cases}
    \end{align*}
\end{theorem}
Finally, let us give the historical background of the complex Monge-Amp\`ere flows before we finish the introduction. In 1985, Cao \cite{Cao85} reproved the Yau's solution to the Calabi conjecture and existence of K\"ahler-Einstein metric when the first Chern class is negative or zero by using the K\"ahler-Ricci flow. Since then, the K\"ahler-Ricci flow has been an important topic in itself with various research subjects, and we refer the including its behavior on Fano manifolds \cite{PSSW, PS06, Sz10}, extension on non-K\"ahler settings \cite{Gi11, TW15}, and so on. We refer the reader to survey papers \cite{SW13, To18} for more details. \\
\indent It is a classical result that solving K\"ahler-Ricci flow is equivalent to solving a parabolic PDE named complex Monge-Amp\`ere flows. After Cao's result, various parabolic Monge-Amp\`ere equations have been extensively studied, including Monge-Amp\`ere flows, J-flows, inverse Monge-Amp\`ere flows and so on. Recently, Phong and T\^o \cite{PT21} introduced the notion of parabolic $C$-subsolutions to study the solvability condition for a large class of parabolic Monge-Amp\`ere equations on compact Hermitian manifolds. Moreover, Picard and Zhang \cite{PZ20} showed the existence and long-time behavior of the solution to the more general parabolic Monge-Amp\`ere equations on compact K\"ahler manifold without the convexity or concavity condition on the operator, which is extended to the compact Hermitian manifold case by Smith \cite{Sm20}.\\
    \indent For the complex Monge-Amp\`ere flows, studying the weak solutions of degenerate equations has been a main interest for its geometrical application, especially for the analytic minimal model program proposed by Song and Tian \cite{ST17}. Eyssidieux, Guedj and Zeriahi \cite{EGZ15, EGZ16} first studied the viscosity approach. Later, Guedj, Lu and Zeriahi \cite{GLZ20, GLZ21-2} introduced the pluripotential approach. They also proved in \cite{GLZ21-1} that both approaches are equivalent under some assumptions. Recently, Dang \cite{Dan24} also extended the pluripotential approach to the Chern-Ricci flow to study the question proposed by Tosatti and Weinkove \cite{TW22}. We refer the reader to \cite{To21, Da22} for the application of viscosity and pluripotential approaches. Besides its geometrical applications, degenerate complex Monge-Amp\`ere flow has been also an interesting topic itself. We refer the reader to \cite{DLT20, DP22} and \cite{Do16, Do17b} for various generalizations of viscosity approach and more general initial data, respectively.\\
    \mbox{}\\
\indent \textbf{Organization.} In Section $2$, we recall the definition of parabolic Monge-Amp\`ere operator in \cite{GLZ21-2}. We also prove several useful lemmas about the operator. In Section $3$, we first prove some lemmas about mixed type inequalities and property of supersolutions. Then we give the proof of Theorem \ref{main result 1} in Theorem \ref{theorem 3.8} by using these lemmas. With this result, we also prove the uniqueness of the weak solution to the Cauchy-Dirichlet problem and study the long-term behavior of the solution. \\
\mbox{}\\
\indent \textbf{Acknowledgements.} I would like to thank my advisor, Ngoc Cuong Nguyen, for providing invaluable guidance and support. I would like to thank Quan Tuan Dang for helpful discussion on the proof of the Theorem \ref{theorem 3.8} and to Hyunsoo Ahn for a careful reading of the earlier version of its proof. I also thank to Do Hoang Son and S{\l}awmoir Dinew for helpful comments.

\section{Parabolic Monge-Amp\`ere Operator}
In this section, we prove a comparison principle, a global maximum principle, and a domination principle of the parabolic Monge-Amp\`ere operator. We first recall the definition of the operator in \cite{GLZ21-2}.
\begin{definition}\label{parabolic MA operator}{\cite[Definition 2.2]{GLZ21-2}}
    Let $u \in \mathcal{P}(\Omega_T) \cap L^{\infty}_{loc}(\Omega_T)$. The map
    \begin{align*}
        C_c(\Omega_T) \ni \chi \mapsto \int_{\Omega_T}\chi(t, z) dt \wedge (dd^cu)^n := \int_0^T dt \left(\int_{\Omega}\chi(t, \cdot)(dd^cu(t, \cdot))^n\right)
    \end{align*}
    defines a positive Radon measure in $\Omega_T$ denoted by $dt \wedge (dd^cu)^n$.
\end{definition}
Note that the integration in the above definition is well-defined since it was proved in \cite[Lemma 2.1]{GLZ21-2} that for $u \in \mathcal{P}(\Omega_T) \cap L^{\infty}_{loc}(\Omega_T)$ and $\chi \in C_c(\Omega_T)$, the map
\begin{align*}
    \Gamma_\chi : t \mapsto \int_{\Omega}\chi(t, \cdot)(dd^cu(t, \cdot))^n 
\end{align*}
is continuous in $(0, T)$. We first prove a lemma which extends this result.
\begin{lemma}\label{Et}
    Let $u \in \mathcal{P}(\Omega_T) \cap L^{\infty}_{loc}(\Omega_T)$ and $E \subset \Omega_T$ be a Borel subset. The map
    \begin{align*}
        \Gamma_{E} : t \mapsto \int_{E_t}(dd^cu(t, \cdot))^n
    \end{align*}
    is Borel measurable in $(0, T)$, where $E_t := \{z \in \Omega ~\mid~ (t, z) \in E\}$.
\end{lemma}
\begin{proof}
    We first assume that $E$ is compact. Let $\{\chi_j(t, z)\}_{j \geq 1}$ be a sequence of continuous positive test functions in $\Omega_T$ such that $\chi_j \downarrow \mathbf{1}_{E}$ as $j \rightarrow \infty$. Since $\chi_j(t, \cdot) \downarrow \mathbf{1}_{E_t}$ as $j \rightarrow \infty$ for all $t \in (0, T)$, we have $\Gamma_{\chi_j} \downarrow \Gamma_{E}$ as $j \rightarrow \infty$. Therefore $\Gamma_{E}$ is Borel measurable. Next, assume that $E$ is open. Let $\{E_j\}_{j \geq 1}$ be a sequence of compact subsets of $\Omega_T$ such that $E_j \uparrow E$ as $j \rightarrow \infty$. Since $\Gamma_{E_j}$ is a Borel measurable function for all $j$ and $\Gamma_{E_j} \uparrow \Gamma_{E}$ as $j \rightarrow \infty$, $\Gamma_{E}$ is also Borel measurable. \\
    \indent Finally, we prove the general case. Let us fix open subsets $I \Subset (0, T)$ and $U \Subset \Omega$. Since $u \in \mathcal{P}(\Omega_T) \cap L^{\infty}_{loc}(\Omega_T)$, it follows from the elliptic Chern-Levine-Nirenberg inequality (see e.g. \cite[Theorem 3.9]{GZ17}) that
    \begin{align}\label{uniform finiteness}
        \sup_{t \in I}\Gamma_{I \times U}(t) = \sup_{t \in I}\int_{U}(dd^cu(t, \cdot))^n < +\infty.
    \end{align}
    Let $\mathcal{M}$ be a family of all Borel subsets $E$ in $\Omega_T$ such that $\Gamma_{E \cap (I \times U)}(t)$ is a Borel measurable function on $(0, T)$. Then $\mathcal{M}$ satisfies the following properties 
    \begin{itemize}
        \item [(i)] $\mathcal{M}$ contains all the open sets in $\Omega_T$,
        \item [(ii)] If $E, F \in \mathcal{M}$ and $F \subset E$, then $E \setminus F \in \mathcal{M}$,
        \item [(iii)] $\mathcal{M}$ is closed under finite disjoint unions,
        \item [(iv)] $\mathcal{M}$ is closed under countable increasing unions and countable decreasing intersections.
    \end{itemize}
    Since (iii) and (iv) are straightforward, we only explain (i) and (ii). Indeed, (i) is true since if $E \subset \Omega_T$ is open, then $E \cap (I \times U)$ is open. Next, if $E, F \in \mathcal{M}$ and $F \subset E$, then
    \begin{equation}\label{difference}
    \begin{aligned}
        \Gamma_{(E \setminus F) \cap (I \times U)}(t) &= \int_{(E\setminus F)_t \cap U}(dd^cu(t, \cdot))^n\\ &= \int_{E_t \cap U}(dd^cu(t, \cdot))^n - \int_{F_t \cap U}(dd^cu(t, \cdot))^n \\
        &= \Gamma_{E \cap (I \times U)}(t) - \Gamma_{F \cap (I \times U)}(t).
    \end{aligned}    
    \end{equation}
        It implies that $E \setminus F \in \mathcal{M}$. Note that we used (\ref{uniform finiteness}) for the second equality of (\ref{difference}). Therefore by proceeding as the proof of \cite[Theorem 7.26]{Fo84}, one can prove that any Borel subset $E$ of $\Omega_T$ is contained in $\mathcal{M}$. Now, let $\{I_j\}_{j \geq 1}$ and $\{U_j\}_{h \geq 1}$ be sequence of relatively compact open subsets of $(0, T)$ and $\Omega$ such that $I_j \uparrow (0, T)$ and $U_j \uparrow \Omega$. If $E$ is a Borel subset of $\Omega_T$, then by the preceding arguments, $\Gamma_{E \cap (I_j \times U_j)}(t)$ is a Borel measurable function on $(0, T)$ for all $j$. By taking $j \rightarrow \infty$, we get the conclusion.
\end{proof}
From now, we use the notation $E_t$ for a Borel set $E \subset \Omega_T$ as defined in Lemma \ref{Et} throughout this paper. Note that because of Lemma \ref{Et}, the integration 
\begin{align}\label{integration}
    \int_0^T dt \int_{E_t}(dd^cu(t, \cdot))^n
\end{align}
is well-defined for all Borel subsets $E \subset \Omega_T$. We next show that (\ref{integration}) coincides with the integration of 
$E$ with respect to a Radon measure $dt \wedge (dd^cu)^n$, which is defined in Definition \ref{parabolic MA operator}. We will use this lemma to compute and compare the size of sublevel sets with respect to the measures of the form $dt \wedge (dd^cu)^n$ where $u \in \mathcal{P}(\Omega_T) \cap L^{\infty}(\Omega_T)$.
\begin{lemma}\label{2.3}
    Let $u \in \mathcal{P}(\Omega_T) \cap L^{\infty}_{loc}(\Omega_T)$. For any Borel subset $E \subset \Omega_T$, we have
    \begin{align*}
        \int_{E}dt \wedge (dd^cu)^n = \int_0^T dt \int_{E_t}(dd^cu(t, \cdot))^n.
    \end{align*}
\end{lemma}
\begin{proof}
        Let us first assume that $E$ is compact. Let $J \subset (0, T)$ and $K \subset \Omega$ be compact subsets such that $E \subset J \times K$. It follows from the elliptic Chern-Levine-Nirenberg inequality that         \begin{align}\label{uniform estimate}
            \sup_{t \in (0, T)}\int_{E_t}(dd^cu(t, \cdot))^n \leq \sup_{t \in J}\int_K (dd^cu(t, \cdot))^n < +\infty.
        \end{align}
        Let $\{\chi_{j}(t, z)\}_{j \geq 1}$ be a sequence of continuous positive test functions in $\Omega_T$ such that $supp(\chi_j) \subset J \times K$ for all $j$ and $\chi_j \downarrow \mathbf{1}_{E}$ as $j \rightarrow \infty$. By the definition of $dt \wedge (dd^cu)^n$, 
    \begin{align}\label{1}
        \int_{\Omega_T}\chi_j(t, z) dt \wedge (dd^cu)^n = \int_0^T dt \int_{\Omega}\chi_j(t, \cdot)(dd^cu(t, \cdot))^n.
    \end{align}
    Since $\chi_j(t, \cdot) \downarrow \mathbf{1}_{E_t}$ as $j \rightarrow \infty$ for all $t \in (0, T)$, it follows from (\ref{uniform estimate}) and the dominated convergence theorem that
    \begin{align}\label{2}
        \lim_{j \rightarrow \infty}\int_0^T dt \int_{\Omega}\chi_j(t, \cdot)(dd^cu(t, \cdot))^n = \int_0^T dt \int_{E_t}(dd^cu(t, \cdot))^n.
    \end{align}
    Combining (\ref{1}) and (\ref{2}), we have
    \begin{align*}
        \int_{E}dt \wedge (dd^cu)^n = \lim_{j \rightarrow \infty}\int_{\Omega_T}\chi_j(t, z) dt \wedge (dd^cu)^n = \int_0^T dt \int_{E_t}(dd^cu(t, \cdot))^n. 
    \end{align*} 
    Next, for an open subset $E \subset \Omega_T$, let $\{E_j\}_{j \geq 1}$ be a sequence of compact subsets of $\Omega_T$ such that $E_j \uparrow E$ as $j \rightarrow \infty$. It follows from the monotone convergence theorem that
    \begin{align*}
         \lim_{j \rightarrow \infty}\int_{E_j}dt \wedge (dd^cu)^n = \int_Edt \wedge (dd^cu)^n
    \end{align*}
    and
    \begin{align*}
        \lim_{j \rightarrow \infty}\int_{0}^T dt \int_{(E_j)_t}(dd^cu(t, \cdot))^n= \int_0^T dt \int_{E_t}(dd^cu(t, \cdot))^n.
    \end{align*}
    Therefore we get the result since we have
    \begin{align*}
        \int_{E_j}dt \wedge (dd^cu)^n = \int_0^Tdt \int_{(E_j)_t}(dd^cu(t, \cdot))^n
    \end{align*}
    for all $j$. Finally, for a Borel subset $E \subset \Omega_T$, one can prove the result by using inner and outer regularity of the Radon measure $dt \wedge (dd^cu)^n$.
\end{proof}
We now prove a parabolic analogue of the classical comparison principle. 
\begin{lemma}\label{CP1}
    Let $u, v \in \mathcal{P}(\Omega_T) \cap L^{\infty}(\Omega_T)$ satisfying $\liminf_{(t, z) \rightarrow [0, T) \times \partial \Omega}(v(t, z) - u(t, z)) \geq 0$. Then
    \begin{align*}
        \int_{\{v < u\}}dt \wedge (dd^cu)^n \leq \int_{\{v < u\}}dt \wedge (dd^cv)^n.
    \end{align*}
\end{lemma}
\begin{proof}
    Fix $\varepsilon_1 > 0$ and $0 < T' < T$. Since $\liminf_{(t, z) \rightarrow [0, T'] \times \partial \Omega}(v(t, z) - u(t, z)) \geq 0$, there exists a compact set $K \subset \Omega$ such that 
    \begin{align*}
        \{v<u-\varepsilon_1\} \cap ((0, T') \times \Omega) \subset [0, T'] \times K.
    \end{align*}
    Let $\chi_{\varepsilon_2}(z)$ be a standard smooth mollifier and define
    \begin{align*}
        u_{\varepsilon_2}(t, z) := u(t, z) * \chi_{\varepsilon_2}(z) = \int_{\Omega}u(t, \zeta)\chi_{\varepsilon_2}(z-\zeta)dV(\zeta).
    \end{align*}
    For sufficiently small $\varepsilon_2 > 0$, $u_{\varepsilon_2}$ is defined on $(0, T) \times \Omega'$ for some open set $\Omega'$ satisfying $K \Subset \Omega' \Subset \Omega$. Moreover, one can check that $u_{\varepsilon_2} \in \mathcal{P}((0, T) \times \Omega') \cap C((0, T) \times \Omega')$ and $u_{\varepsilon_2} \downarrow u$ as $\varepsilon_2 \rightarrow 0$. Since $v$ is (jointly) upper semi-continuous by \cite[Proposition 1.4]{GLZ21-2}, $\{v < u_{\varepsilon_2}-\varepsilon_1\} \cap ((0, T') \times \Omega')$ is an open subset of $\Omega_T$ which decreases to $\{v < u - \varepsilon_1\} \cap ((0, T') \times \Omega)$ as $\varepsilon_2 \rightarrow 0$. By taking $\varepsilon_1 \rightarrow 0$ and $T' \rightarrow T$, one can see that $\{v < u\}$ is a Borel subset of $\Omega_T$.\\
    \indent It follows from the classical comparison principle (see e.g. \cite[Theorem 3.29]{GZ17}) that for all $t \in (0, T)$,
    \begin{align}\label{inequality for integration}
        \int_{\{v(t, \cdot) < u(t, \cdot)\}}(dd^cu(t, \cdot))^n \leq \int_{\{v(t, \cdot) < u(t, \cdot)\}}(dd^cv(t, \cdot))^n.
    \end{align}
    It follows from Lemma \ref{Et} that each side of (\ref{inequality for integration}) is a Borel measurable function on $(0, T)$. By integrating each side with respect to $dt$ for $t \in (0, T)$,
    \begin{align*}
        \int_{0}^{T} dt \int_{\{v(t, \cdot) < u(t, \cdot)\}}(dd^cu(t, \cdot))^n \leq \int_{0}^{T} dt \int_{\{v(t, \cdot) < u(t, \cdot)\}}(dd^cv(t, \cdot))^n.
    \end{align*}
    It follows from Lemma \ref{2.3} that
    \begin{align*}
        \int_{\{v<u\} }dt \wedge (dd^cu)^n \leq \int_{\{v < u\}}dt \wedge (dd^cv)^n,
    \end{align*}
    which completes the proof.
\end{proof}
The following two lemmas are parabolic analogues of classical global maximum principle and domination principle. The proofs are almost same with the ones in \cite[Corollary 3.30, Corollary 3.31]{GZ17}.
\begin{lemma}
    Let $u, v \in \mathcal{P}(\Omega_T) \cap L^{\infty}(\Omega_T)$ satisfying $\liminf_{(t, z) \rightarrow [0, T) \times \partial \Omega}(v(t, z) - u(t, z)) \geq 0$. If $dt \wedge (dd^cv)^n \leq dt \wedge (dd^cu)^n$, then $u \leq v$ in $\Omega_T$.
\end{lemma}
\begin{proof}
    For $\varepsilon > 0$, we set $\phi_{\varepsilon} := u + \varepsilon \rho$, where $\rho(t, z) := \lvert z\rvert^2 - R^2$ for $R > 0$ large enough so that $\rho < 0$ on $\Omega_T$. Then $\phi_{\varepsilon} \in \mathcal{P}(\Omega_T) \cap L^{\infty}(\Omega_T)$ and $\liminf_{z \rightarrow \partial \Omega}(v(t, z) - \phi_{\varepsilon}(t, z)) \geq 0$ for all $t \in (0, T)$. It follows from Lemma \ref{CP1} that 
    \begin{align*}
        \int_{\{v < \phi_{\varepsilon}\}}dt \wedge (dd^c{\phi}_{\varepsilon})^n \leq \int_{\{v < \phi_{\varepsilon}\}}dt\wedge (dd^cv)^n.
    \end{align*}
    Since
    \begin{align*}
        dt \wedge (dd^c\phi_{\varepsilon})^n \geq dt \wedge (dd^cv)^n + \varepsilon^n dt \wedge (dd^c\rho)^n,
    \end{align*}
    we have $\int_{\{v < \phi_{\varepsilon}\}}dt \wedge (dd^c\rho)^n = 0$. It implies that the set $\{v < \phi_{\varepsilon}\}$ has Lebesgue measure zero, so that $\{v < u\}$ also has Lebesgue measure $0$. Therefore by \cite[Corollary 1.9]{GLZ21-2}, we have $u \leq v$ in $\Omega_T$. 
\end{proof}
\begin{lemma}\label{domination principle}
    Let $u, v \in \mathcal{P}(\Omega_T) \cap L^{\infty}(\Omega_T)$ satisfying $\liminf_{(t, z) \rightarrow [0, T) \times \partial \Omega}(v(t, z) - u(t, z)) \geq 0$. Assume that for all $u \leq v$ a.e. in $\Omega_T$ with respect to the measure $dt \wedge (dd^cv)^n$. Then $u \leq v$ in $\Omega_T$.
\end{lemma}
\begin{proof}
    For $\varepsilon > 0$, we set $u_{\varepsilon} := \phi_{\varepsilon} + \varepsilon \rho$, where $\rho(t, z) := \lvert z\rvert^2 - R^2$ for $R > 0$ large enough so that $\rho < 0$ on $\Omega_T$. By Lemma \ref{CP1}, 
    \begin{align*}
        \int_{\{v < \phi_{\varepsilon}\}}dt \wedge (dd^c\phi_{\varepsilon})^n \leq \int_{\{v < \phi_{\varepsilon}\}}dt \wedge (dd^cv)^n \leq \int_{\{v < u\}}dt \wedge (dd^cv)^n = 0.
    \end{align*}
    Since $dt \wedge (dd^c\phi_{\varepsilon})^n \geq \varepsilon^n dt \wedge (dd^c\rho)^n$, it implies that the set $\{v < \phi_{\varepsilon}\}$ has Lebesgue measure $0$, so that $\{v < u\}$ also has Lebesgue measure $0$. Therefore by \cite[Corollary 1.9]{GLZ21-2}, we have $u \leq v$ in $\Omega_T$.
\end{proof}
\section{A Generalized Comparison Principle}
In this section, we give a proof of Theorem \ref{main result 1}. We first prove some useful lemmas which will be used to prove it.
\subsection{Mixed type inequalities}
We show the following lemmas on mixed type inequalities. The main result of \cite{GLZ19} was used to characterize a subsolution in the last part of the proof of \cite[Theorem 6.6]{GLZ21-2}, which is not available in our case. The following lemmas will be used instead. 
    \begin{lemma}\label{Lemma 3.1}
    Assume that $\varphi, \psi \in PSH(\Omega) \cap L^{\infty}(\Omega)$ satisfy
    \begin{align*}
        (dd^c\varphi)^n \geq f(dd^c\psi)^n
    \end{align*}
    for some $f \in L^1(\Omega, (dd^c\psi)^n)$. Then, we have
    \begin{align*}
        (dd^c\varphi) \wedge (dd^c\psi)^{n-1} \geq f^{1/n}(dd^c\psi)^n.
    \end{align*}
\end{lemma}
\begin{proof}
    It follows directly from Dinew's mixed type inequality in \cite[Theorem 1.3]{Di09}. 
\end{proof}
\begin{lemma}\label{Lemma 3.2}
    Assume that $u, \varphi \in PSH(\Omega) \cap L^{\infty}(\Omega)$ satisfy $\liminf_{z \rightarrow \partial \Omega}(\varphi(z)-u(z)) \geq 0$. Then
    \begin{align*}
        \int_{\{\varphi< u\}}(dd^cu) \wedge (dd^c\varphi)^{n-1} \leq \int_{\{\varphi < u\}}(dd^c\varphi)^n.
    \end{align*}
\end{lemma}
\begin{proof}
    The proof is almost same with the one for the comparison principle by Bedford and Taylor in \cite[Theorem 4.2]{BT82} (see also \cite[Theorem 1.16]{Ko05}).
\end{proof}
\subsection{Boundary behavior of supersolutions}
Next, we show the following lemma about the property of supersolutions. Let $h$ be a Cauchy-Dirichlet boundary data and denote $v_0(z) := h(0, z)$ for $z \in \Omega$.
\begin{lemma}\label{Lemma 3.7}
    Let us assume that $v \in \mathcal{P}(\Omega_T) \cap L^{\infty}(\Omega_T)$ satisfy
    \begin{align*}
        \begin{cases}
            &dt \wedge (dd^cv)^n \leq e^{\partial_tv+F(t, z, v)}dt \wedge d\mu, \\
            &\lim_{(t, z) \rightarrow (\tau, \zeta) \in [0, T ) \times\partial \Omega}v(t, z) = h(\tau, \zeta),\\
            &\lim_{t \rightarrow 0+}v(t, \cdot) = v_0 \text{ in } L^1(d\mu).
        \end{cases}
    \end{align*}
    Then there exists a function $\eta(t) \in C([0, T))$ such that $\eta(0) = 0$ and
        \begin{align*}
            v(t, z) \geq v_0(z) - \eta(t)
        \end{align*}
        for $dt \wedge d\mu$-a.e. $(t, z) \in \Omega_T$.
\end{lemma}
\begin{proof}
We proceed as the proof of \cite[Lemma 6.9]{GLZ21-2}. Let us fix $0 < S < T$. For $s > 0$ small enough, we set
\begin{align*}
    \delta(s) := \sup\{\lvert h(\tau, z) - h(t, z)\rvert ~\mid~z \in \partial\Omega, ~t, \tau \in [0, S], ~\lvert t-\tau\rvert \leq s\}.
\end{align*}
Since $h$ is continuous on $[0, T) \times \partial \Omega$, $\lim_{s \rightarrow 0+}\delta(s) = 0$. \\
\indent Now, fix $0 < s < \frac{S}{2}$. We show that for some constant $C > 0$,
\begin{align*}
    v(s, z) \geq v_0(z) - \delta(s) + s(\psi(z)-C) + n(s\log(s/T) - s)
\end{align*}
holds for $\mu$-a.e. $z \in \Omega$. Note that $\psi$ is given in (\ref{MA2}).\\
\indent Fix $0 < \varepsilon \leq s$. Let us define a map $(t, z) \mapsto U^{\varepsilon}(t, z)$ on $\Omega_s$ by
\begin{align*}
    U^{\varepsilon}(t, z) := v(\varepsilon, z) - \delta(s) + t(\psi(z) - C_1) + n(t\log(t/T)-t)
\end{align*}
for some constant $C_1 > 0$ so that $U^{\varepsilon}$ satisfies
\begin{align*}
    dt \wedge (dd^cU^{\varepsilon})^n \geq e^{\partial_tU^{\varepsilon}+F(t, z, U^{\varepsilon})}dt \wedge d\mu \text{ in } \Omega_s. 
\end{align*}
Next, let us define a map $(t, z) \mapsto V^{\varepsilon}(t, z)$ on $\Omega_s$ by 
\begin{align*}
    V^{\varepsilon}(t, z) := v(t+\varepsilon, z) + C_2\varepsilon t
\end{align*}
for some constant $C_2 > 0$ so that $V^{\varepsilon}$ satisfies
\begin{align*}
    dt \wedge (dd^cV^{\varepsilon})^n \leq e^{\partial_tV^{\varepsilon}+F(t, z, V^{\varepsilon})}dt \wedge d\mu \text{ in } \Omega_s.
\end{align*}
Note that $U^{\varepsilon}$ and $V^{\varepsilon}$ satisfy the following properties
\begin{itemize}
    \item [(i)] $\liminf_{z \rightarrow \partial \Omega}(V^{\varepsilon}(t, z)-U^{\varepsilon}(t, z)) \geq 0$ for all $t \in (0, T)$,
    \item [(ii)] $\lim_{t \rightarrow 0+}U^{\varepsilon}(t, z) = v(\varepsilon, z) - \delta(s)$ uniformly for all $z \in \Omega$,
    \item [(iii)] $\lim_{t \rightarrow 0+}V^{\varepsilon}(t, z) = v(\varepsilon, z)$ uniformly for all $z \in \Omega$.
\end{itemize}
Also, $U^{\varepsilon}$ and $V^{\varepsilon}$ are locally uniformly semi-concave in $(0, s)$, and 
\begin{align*}
    \partial_tU^{\varepsilon}(t, z) = \psi(z)-C_1+n\log(t/T),
\end{align*}
which is locally uniformly Lipschitz in $(0, T)$. Therefore it follows from Lemma \ref{Lemma 3.3} that $V^{\varepsilon}(t, z) \geq U^{\varepsilon}(t, z)$ for all $(t, z) \in \Omega_s$. Recall that $\lim_{t \rightarrow 0+}v(t, \cdot) = v_0$ in $L^1(d\mu)$. Hence there exists a sequence $\{\varepsilon_k\}_{k \geq 1}$ such that $\varepsilon_k \downarrow 0$ as $k \rightarrow \infty$ and $v(t_k, z) \rightarrow v(0, z)$ as $k \rightarrow \infty$ for $d\mu$-a.e. $z \in \Omega$. Hence after substituting $\varepsilon$ by $\varepsilon_k$ and taking $k \rightarrow \infty$, we get
\begin{align*}
    v(t, z) \geq v_0(z) - \delta(s) + t(\psi(z)-C_1)+n(t\log(t/T)-t)
\end{align*}
holds for all $t \in (0, s)$ and $\mu$-a.e. $z \in \Omega$. By taking $t \rightarrow s$, we get the result.
\end{proof}
    \subsection{Main result}
We now introduce a lemma which plays a key role in this paper. Recall that for a bounded parabolic potential $u$ which is locally uniformly semi-concave in $(0, T)$, 
\begin{align*}
    \partial_t^+u(t, z) := \lim_{\delta \rightarrow 0+}\frac{u(t+\delta, z)-u(t, z)}{\delta}
\end{align*}
and
\begin{align*}
    \partial_t^-u(t, z) := \lim_{\delta \rightarrow 0-}\frac{u(t+\delta, z)-u(t, z)}{\delta}
\end{align*}
exist for all $(t, z) \in \Omega_T$ and $\partial_tu(t, z)$ exists when $\partial_t^+u(t, z) = \partial_t^-u(t, z)$.
\begin{lemma}\label{Lemma 3.3}
      Assume that $u, v \in \mathcal{P}(\Omega_T) \cap L^{\infty}(\Omega_T)$ satisfy 
    \begin{enumerate}
        \item [(a)] $\liminf_{(t, z) \rightarrow [0, T) \times \partial\Omega}(v(t, z) - u(t, z)) \geq 0$,
        \item [(b)] for each $\varepsilon > 0$, there exists $\Tilde{\varepsilon} > 0$ such that
        \begin{align*}
            \int_{\{v +(t+1)\varepsilon< u\} \cap ((0, \Tilde{\varepsilon}) \times \Omega)}dt \wedge d\mu = 0,
        \end{align*}
        \item [(c)] $u$ and $v$ are locally uniformly semi-concave in $(0, T)$,
        \item [(d)] for a.e. $t \in (0, T)$,
        \begin{align*}
            &(dd^cu(t, \cdot)) \wedge (dd^cv(t, \cdot))^{n-1} \\&\quad\geq \exp\left(\frac{\partial_tu(t,\cdot)-\partial_tv(t, \cdot)+F(t, \cdot, u)-F(t, \cdot, v)}{n}\right)(dd^cv(t, \cdot))^n,
        \end{align*}
        \item [(e)] $dt \wedge (dd^cv)^n \leq e^{\partial_tv+F(t, z, v)}dt \wedge d\mu$ in $\Omega_T$,
        \item [(f)] $\partial_tu(t, z)$ is well-defined for all $(t, z) \in \Omega_T$ and locally uniformly Lipschitz in $(0, T)$.
    \end{enumerate}
    Then $u \leq v$ in $\Omega_T$.
\end{lemma}
\begin{proof}
    Fix $\varepsilon> 0$. Let us define
    \begin{align*}
        E := \{v+(t+1)\varepsilon < u\}. 
    \end{align*}
    Our goal is to show that $E = \emptyset$. Assume that it is not the case. Let us denote $E_t := \{z \in \Omega ~\mid~ (t, z) \in E\}$. We define 
    \begin{align*}
        t_1 := \inf\{t \in (0, T)~\mid~ \mu(E_t) > 0\}.
    \end{align*}
    The following Figure 1 describes our strategy of the proof. 
    \begin{figure}[hbt]
        \centering
        \includegraphics[width=0.7\linewidth]{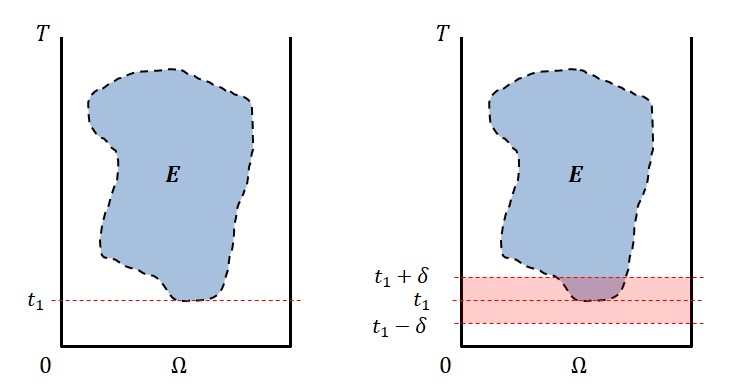}
        \label{figure_1}
        \caption{}
    \end{figure}
    \noindent The idea is to prove that $t_1 > 0$ as in the left figure, and then compute the size of the set $E \cap((t_1-\delta, t_1+\delta) \times \Omega )$ which is described in the right figure, where $\delta > 0$ will be chosen later. We will show that it is actually an empty set, which contradicts to the definition of $t_1$. \\
    \indent We first show that $t_1 > 0$. By the assumptions (b) and (e), there exists $\Tilde{\varepsilon} > 0$ such that
    \begin{align*}
        0 \leq \int_{E \cap ((0, \Tilde{\varepsilon}) \times \Omega)}dt \wedge (dd^cv)^n \leq \int_{E \cap ((0, \Tilde{\varepsilon}) \times \Omega)}e^{\partial_tv+F(t, z, v)}dt \wedge d\mu = 0,
    \end{align*}
    which implies that
    \begin{align*}
        \int_{E \cap ((0, \Tilde{\varepsilon}) \times \Omega)}dt \wedge (dd^cv)^n = 0.
    \end{align*}
    It follows from Lemma \ref{domination principle} that $E \cap ((0, \Tilde{\varepsilon}) \times \Omega) = \emptyset$. Therefore $\mu(E_t) = 0$ for all $t \in (0, \Tilde{\varepsilon})$, which implies that $t_1 \geq \Tilde{\varepsilon} > 0$.
    \\ 
    \indent Next, by the assumption (f), there exists $C_1 > 0$ such that
    \begin{align*}
        \lvert \partial_tu(t, z) - \partial_tu(s, z)\rvert < C_1\lvert t-s\rvert
    \end{align*}
    for all $t, s \in \left[t_1-\frac{\Tilde{\varepsilon}}{2}, t_1+\frac{\Tilde{\varepsilon}}{2}\right]$ and $z \in \Omega$. It follows from the assumption (c) that there exists $C_2> 0$ such that $u(t, z) - C_2t^2$ and $v(t, z) - C_2t^2$ are concave in $\left[t_1-\frac{\Tilde{\varepsilon}}{2}, t_1+\frac{\Tilde{\varepsilon}}{2}\right]$ for all $z \in \Omega$. Let us define $C_3 := \max\{C_1, C_2\}$, so that we have
    \begin{align}\label{lipschitzness of derivative}
        \lvert \partial_tu(t, z) - \partial_tu(s, z)\rvert < C_3\lvert t-s\rvert
    \end{align}
    for all $t, s \in \left[t_1-\frac{\Tilde{\varepsilon}}{2}, t_1+\frac{\Tilde{\varepsilon}}{2}\right]$ and $z \in \Omega$, and 
    \begin{align}\label{concavity of u and v}
        u(t, z) - C_3t^2, v(t, z) - C_3t^2 \text{ are concave in } \left[t_1-\frac{\Tilde{\varepsilon}}{2}, t_1+\frac{\Tilde{\varepsilon}}{2}\right]
    \end{align}
    for all $z \in \Omega.$\\
    \indent Let us denote $\delta := \min\{\frac{\Tilde{\varepsilon}}{4}, \frac{\varepsilon}{6C_3}\}$ and fix $t \in (t_1-\delta, t_1+\delta)$. Since $\mu(E_{t-\delta}) = 0$, we have
    \begin{align*}
        u(t, z) > v(t, z) + (t+1)\varepsilon \text{ for all } z \in E_t
    \end{align*}
    and
    \begin{align*}
        u(t-\delta, z) \leq v(t-\delta, z) + (t-\delta+1)\varepsilon \text{ for $\mu$-a.e. }z \in E_t.
    \end{align*}
    Therefore we have
    \begin{align}\label{3.3}
        \frac{u(t, z) - u(t-\delta, z)}{\delta} > \frac{v(t, z) - v(t-\delta, z)}{\delta}+\varepsilon \text{ for $\mu$-a.e. $z \in E_t$.} 
    \end{align}
    It follows from (\ref{concavity of u and v}) that
    \begin{align*}
        \frac{(u(t, z) - C_3t^2)-(u(t-\delta, z)-C_3(t-\delta)^2)}{\delta}&\leq \partial_t^-u(t-\delta, z) - 2C_3(t-\delta)\\&\quad = \partial_tu(t-\delta, z) - 2C_3(t-\delta)
    \end{align*}
    and
    \begin{align*}
        \frac{(v(t, z)-C_3t^2)-(v(t-\delta, z)-C_3(t-\delta)^2)}{\delta}\geq \partial_t^-v(t, z) - 2C_3t
    \end{align*}
    for all $z \in E_t$. Hence one may induce from (\ref{3.3}) that
    \begin{align}\label{3.4}
        \partial_tu(t-\delta, z) + 2C_3\delta > \partial_t^-v(t, z) + \varepsilon \text{ for $\mu$-a.e. $z \in E_t$.}
    \end{align}
    Note that $t$ and $t-\delta$ are in $\left(t_1-\frac{\Tilde{\varepsilon}}{2}, t_1+\frac{\Tilde{\varepsilon}}{2}\right)$. By combining (\ref{lipschitzness of derivative}) and (\ref{3.4}), we get
    \begin{align}\label{comparing time derivatives of u and v}
        \partial_tu(t, z) > \partial_t^-v(t, z) + \frac{\varepsilon}{2} \text{ for $\mu$-a.e. $z \in E_t$.}
    \end{align}
    \indent We now show that $E \cap ((t_1-\delta, t_1+\delta) \times \Omega) = \emptyset$. For a.e. $t \in (t_1-\delta, t_1+\delta)$, we have
    \begin{equation}\label{fiberwise inequality}
        \begin{aligned}
        &\int_{E(t, \cdot)}e^{\varepsilon/2n}(dd^cv(t, \cdot))^n \\
        &\quad\leq \int_{E(t, \cdot)}\exp\left(\frac{\partial_tu(t, \cdot)-\partial_tv(t, \cdot)+F(t, \cdot, u)-F(t, \cdot, v)}{n}\right)(dd^cv(t, \cdot))^n \\
        &\quad\leq \int_{E(t, \cdot)}(dd^cu(t, \cdot)) \wedge (dd^cv(t, \cdot))^{n-1} \\
        &\quad \leq \int_{E(t, \cdot)}(dd^cv(t, \cdot))^n,
        \end{aligned}
    \end{equation}
    where we used (\ref{comparing time derivatives of u and v}) for the first inequality and Lemma \ref{Lemma 3.2} for the last inequality. Indeed, we could use (\ref{comparing time derivatives of u and v}) since 
    \begin{align*}
        (dd^cv(t, \cdot))^n \leq e^{\partial_tv(t, \cdot)+F(t, \cdot, v)}d\mu
    \end{align*}
    for a.e. $t \in (0, T)$ by \cite[Proposition 3.2]{GLZ21-2}, which implies that for a.e. $t \in (t_1-\delta, t_1+\delta)$, (\ref{comparing time derivatives of u and v}) holds for $(dd^cv(t, \cdot))^n$-a.e. $z \in E_t$. By integrating both sides of  (\ref{fiberwise inequality}) with respect to $t$ for $(t_1-\delta, t_1+\delta)$, we  get
    \begin{align*}
        0 \leq e^{\varepsilon/2n}\int_{t_1-\delta}^{t_1+\delta}dt \int_{E(t, \cdot)}(dd^cv(t, \cdot))^n \leq \int_{t_1-\delta}^{t_1+\delta}dt \int_{E(t, \cdot)}(dd^cv(t, \cdot))^n,
    \end{align*}
    which implies that
    \begin{align*}
        \int_{t_1-\delta}^{t_1+\delta}dt \int_{E(t, \cdot)}(dd^cv(t, \cdot))^n = 0.
    \end{align*}
    It follows from Lemma \ref{2.3} that
    \begin{align*}
        \int_{E \cap ((t_1-\delta, t_1+\delta) \times \Omega)}dt \wedge (dd^cv)^n = \int_{t_1-\delta}^{t_1+\delta}dt\int_{E(t, \cdot)}(dd^cv(t, \cdot))^n = 0.
    \end{align*}
    This implies that $E \cap ((t_1-\delta, t_1+\delta) \times \Omega) = \emptyset$ by Lemma \ref{domination principle}. Since it contradicts to the definition of $t_1$, we conclude that $E = \emptyset$.
    \end{proof}
    \begin{remark}
    We provide an example for which the assumptions of the above lemma are satisfied. First, $u$ and $v$ satisfy the condition (b) if $\lim_{t \rightarrow 0+}v(t, z) = v_0(z)$ and $\lim_{t \rightarrow 0+}u(t, z) = u_0(z)$ uniformly where $v_0, u_0 \in PSH(\Omega)\cap L^{\infty}(\Omega)$ satisfy $v_0 \geq u_0$. Indeed, for all $z \in \Omega$, there exists $\delta_1 > 0$ such that if $0 < \lvert t\rvert < \delta_1$, then $\lvert v_0(z) - v(t, z)\rvert < \varepsilon/2$ and $\lvert u_0(z) - u(t, z)\rvert < \varepsilon/2$. Therefore for all $(t, z) \in (0, \delta_1) \times \Omega$, 
    \begin{align*}
        v(t, z) + (t+1)\varepsilon &> v_0(z) + \left(t+\frac{1}{2}\right)\varepsilon \\&\geq 
        u_0(z) + \left(t+\frac{1}{2}\right)\varepsilon\\ &> u(t, z) + t\varepsilon\\ &> u(t, z).
    \end{align*}
    \indent Next, $u$ and $v$ satisfy the condition (d) if $dt \wedge (dd^cu)^n \geq e^{\partial_tu+F(t, z, u)}dt \wedge d\mu$ and $dt \wedge (dd^cv)^n \leq e^{\partial_tv+F(t, z, v)}dt \wedge d\mu$. Indeed, by using \cite[Proposition 3.2]{GLZ21-2}, one can show that for a.e. $t \in (0, T)$, $u$ and $v$ satisfy
    \begin{align*}
        (dd^cu(t, \cdot))^n \geq e^{\partial_tu(t, \cdot)+F(t, \cdot, u)}d\mu \geq e^{\partial_tu(t, \cdot)-\partial_tv(t, \cdot)+F(t, \cdot, u) - F(t, \cdot, v)}(dd^cv(t, \cdot))^n.
    \end{align*}
    It follows from Lemma \ref{Lemma 3.1} that for a.e. $t \in (0, T)$,
    \begin{align*}
        &(dd^cu(t, \cdot)) \wedge (dd^cv(t, \cdot))^{n-1} \\
        &\quad\geq \exp\left(\frac{\partial_tu(t, \cdot)-\partial_tv(t, \cdot)+F(t, \cdot, u)-F(t, \cdot, v)}{n}\right)(dd^cv(t, \cdot))^n.
    \end{align*}
\end{remark}
We next prove a lemma about the condition (b) on Lemma \ref{Lemma 3.3}. This will be used later while proving Theorem \ref{main result 1}.
\begin{lemma}\label{Lemma 3.6}
     Assume that $u, v \in \mathcal{P}(\Omega_T) \cap L^{\infty}(\Omega_T)$ satisfy 
    \begin{enumerate}
        \item [(a)] $\liminf_{(t, z) \rightarrow [0, T) \times \partial\Omega}(v(t, z) - u(t, z)) \geq 0$,
        \item [(b)] there exists a function $\eta(t) \in C([0, T])$ such that $\eta(0) = 0$ and 
        \begin{align*}
            v(t, z) \geq v_0(z) - \eta(t)
        \end{align*}
        for $dt \wedge d\mu$-a.e. $(t, z) \in \Omega_T$,
        \item [(c)] for each open $U \Subset \Omega$, there exists $A > 0$ such that
        \begin{align*}
            \int_{U}u(t, \cdot)d\mu \leq \int_{U}u_0d\mu + At,
        \end{align*}
        \item [(d)] $u_0(z) \leq v_0(z)$ for all $z \in \Omega$.
    \end{enumerate}
    Then for each $\varepsilon > 0$, there exists $\Tilde{\varepsilon} > 0$ such that
    \begin{align*}
        \int_{\{v+(t+1)\varepsilon < u\}\cap ((0,\Tilde{\varepsilon}) \times \Omega)}dt \wedge d\mu = 0.
    \end{align*}
\end{lemma}
\begin{proof}
    We fix $\varepsilon> 0$. Let us define
    \begin{align*}
        E := \{v+(t+1)\varepsilon < u\}. 
    \end{align*}
    Let us fix $0 < S < T$. Let $U \Subset \Omega$ be an open subset satisfying $E \cap ((0, S] \times \Omega) \subset [0, S] \times U$. Since
    \begin{align*}
        1 < \frac{u-v}{(t+1)\varepsilon} 
    \end{align*}
    in $E$, for all $0 < \delta < S$, we have
    \begin{equation}\label{lemma 3.7-1}
        \begin{aligned}
            &\int_{E \cap ((0, \delta) \times U)}dt \wedge d\mu \\&\quad\leq \int_{E \cap ((0, \delta) \times U)} \frac{u-v}{(t+1)\varepsilon}dt \wedge d\mu \\ &\quad\leq \int_{E \cap ((0, \delta) \times U)}\frac{u-u_0}{(t+1)\varepsilon} dt \wedge d\mu + \int_{E \cap ((0, \delta) \times U)}\frac{v_0-v}{(t+1)\varepsilon}dt \wedge d\mu \\
        &\quad\leq \frac{A}{\varepsilon}\int_{E \cap ((0, \delta)\times U)}\frac{t}{t+1} dt \wedge d\mu + \frac{1}{\varepsilon}\int_{E \cap ((0, \delta) \times U)}\frac{\eta(t)}{t+1}dt \wedge d\mu \\
        &\quad \leq \frac{A\delta+\sup_{t \in (0, \delta)}\lvert \eta(t)\rvert}{\varepsilon} \int_{E \cap ((0, \delta) \times U)}dt \wedge d\mu.
        \end{aligned}
    \end{equation}
    Here, we used the condition (d) for the second inequality and the condition (b) and (c) for the third inequality.
    Let us choose $\Tilde{\varepsilon}$ sufficiently small so that
    \begin{align}\label{lemma 3.7-2}
        A\Tilde{\varepsilon} \leq \frac{\varepsilon}{3} \text{ and } \sup_{t \in (0, \Tilde{\varepsilon})}\lvert \eta(t)\rvert \leq \frac{\varepsilon}{3}.
    \end{align}
    By combining (\ref{lemma 3.7-1}) and (\ref{lemma 3.7-2}), we get
    \begin{align*}
        0 \leq \int_{E \cap ((0, \Tilde{\varepsilon}) \times U)}dt \wedge d\mu \leq \frac{2}{3}\int_{E \cap ((0, \Tilde{\varepsilon})\times U)}dt \wedge d\mu,
    \end{align*}
    which implies that
    \begin{align*}
        \int_{E \cap ((0, \Tilde{\varepsilon}) \times \Omega)}dt \wedge d\mu = \int_{E \cap ((0, \Tilde{\varepsilon}) \times U)}dt \wedge d\mu = 0. 
    \end{align*}
    Thus we get the result.
    \end{proof}
\begin{remark}
    We do not need (\ref{MA}) for the Lemma \ref{Lemma 3.3} and \ref{Lemma 3.6}. They are indeed true for any positive Borel measure $\mu$ satisfying $PSH(\Omega) \subset L^1_{loc}(\Omega, d\mu)$.
\end{remark}
Finally, we are ready to prove our main result.
\begin{theorem}\label{theorem 3.8}
        Let $u, v \in \mathcal{P}(\Omega_T) \cap L^{\infty}(\Omega_T)$ with the Cauchy-Dirichlet boundary data $h_1$, $h_2$ as in (\ref{Cauchy-Dirichlet boundary condition}). Assume that
    \begin{enumerate}
        \item [(a)] $u$ and $v$ are locally uniformly semi-concave in $(0, T)$;
        \item [(b)] $dt \wedge (dd^cu)^n \geq e^{\partial_tu+F(t, z, u)}dt \wedge d\mu$ in $\Omega_T$;
        \item [(c)] $dt \wedge (dd^cv)^n \leq e^{\partial_tv+F(t, z, v)}dt \wedge d\mu$ in $\Omega_T$;
        \item [(d)] $h_1$ satisfies (\ref{additional assumption}).
    \end{enumerate}
    If $h_1 \leq h_2$, then $u \leq v$.
\end{theorem}
\begin{proof}
    Fix $0 < S < T$. We first assume that
    \begin{align}\label{lipschitzness near 0}
        t\lvert \partial_tu(t,z)\rvert \leq B
    \end{align}
    for some $B > 0$ in $\Omega_S$. We will remove this assumption at the end of the proof. For $s > 0$ near $1$ we set, for $(t, z) \in \Omega_S$, 
    \begin{align*}
        W^s(t, z) := s^{-1}u(st, z) - C\lvert s-1\rvert(t+1).
    \end{align*}
    It follows from \cite[Theorem 4.2]{GLZ21-2} that for $C > 0$ large enough, we have 
    \begin{align}\label{property of W^s}
        \begin{cases}
            &dt \wedge (dd^cW^s)^n \geq e^{\partial_tW^{s}+F(t, z, W^s)}dt \wedge d\mu \text{ in } \Omega_S, \\
            &\limsup_{(t, z) \rightarrow (\tau, \zeta) \in [0, S) \times \partial \Omega}W^s(t, z) \leq h_1(\tau, \zeta),\\
            &\limsup_{t \rightarrow 0+}W^s(t, z) \leq h_1(0, z) \text{ for $\mu$-a.e. $z \in \Omega$.} 
        \end{cases}
    \end{align}
    Note that we need the condition (d) to find such a constant $C > 0$. Let $\eta_{\varepsilon}$ be the standard mollifier on $\mathbb{R}$. We set 
    \begin{align*}
        \Phi^{\varepsilon}(t, z) := \int_{\mathbb{R}}W^s(t, z)\eta_{\varepsilon}(s-1)ds.
    \end{align*}
    The idea is to apply the Lemma \ref{Lemma 3.3} for $\Phi^{\varepsilon} - O(\varepsilon)t$ and $v$ in $\Omega_S$, and derive the result. We show that they satisfy all the assumptions given in Lemma \ref{Lemma 3.3}. \\
    \indent One can check that $\partial_t\Phi^{\varepsilon}(t, z)$ is well-defined for all $(t, z) \in \Omega$ and $\Phi^{\varepsilon}$ satisfies
    \begin{align*}
        \begin{cases}
            &\limsup_{(t, z) \rightarrow (\tau, \zeta) \in [0, S) \times \partial \Omega}\Phi^{\varepsilon}(t, z) \leq h_1(\tau, \zeta),\\
            &\limsup_{t \rightarrow 0+}\Phi^{\varepsilon}(t, z) \leq h_1(0, z) \text{ for $\mu$-a.e. $z \in \Omega$.}
        \end{cases}
    \end{align*}
    We prove that $\partial_t\Phi^{\varepsilon}$ is locally uniformly Lipschitz in $(0, T)$. For all $(t, z) \in \Omega$, we have
\begin{align*}
    \partial_t\Phi^{\varepsilon}(t, z) &= \int_{\mathbb{R}}\partial_tW^s(t, z)\eta_{\varepsilon}(s-1)ds \\
    &= \int_{\mathbb{R}}(\partial_{\tau}u(st, z)-C\lvert s-1\rvert)\eta_{\varepsilon}(s-1)ds \\
    &= \int_{\mathbb{R}}t^{-1}\partial_su(st, z)\eta_{\varepsilon}(s-1)ds - C\int_{\mathbb{R}}\lvert s-1\rvert\eta_{\varepsilon}(s-1)ds \\
    &= -\int_{\mathbb{R}}t^{-1}u(st, z)\eta'_{\varepsilon}(s-1)ds - C\int_{\mathbb{R}}\lvert s-1\rvert \eta_{\varepsilon}(s-1)ds.
\end{align*}
Let $J \Subset (0, S)$ be a subinterval. Note that $supp(\eta_{\varepsilon}'(s-1)) \subset (1-\varepsilon_0, 1+\varepsilon_0)$ for some sufficiently small $0 < \varepsilon_0 < 1$. It follows from the local uniform Lipschitzness of $u$ on $(0, T)$ that there exists $\kappa > 0$ such that 
\begin{align}\label{kappa s}
    \lvert u(st_1, z) - u(st_2, z)\rvert \leq \kappa s\lvert t_1-t_2\rvert    
\end{align}
for all $s \in (1-\varepsilon_0, 1+\varepsilon_0)$ and $(t_1, z), (t_2, z) \in J \times \Omega$. Therefore for all $(t_1, z)$, $(t_2, z) \in J\times \Omega$,
\begin{align*}
    &\lvert \partial_t\Phi^{\varepsilon}(t_1, z) - \partial_t\Phi^{\varepsilon}(t_2, z)\rvert \\&= \left \vert \int_{\mathbb{R}}t_1^{-1}u(st_1, z)\eta_{\varepsilon}'(s-1)ds - \int_{\mathbb{R}}t_2^{-1}u(st_2, z)  \eta'_{\varepsilon}(s-1)ds \right\vert \\ 
    &\quad\leq \left\vert\int_{\mathbb{R}} (t_1^{-1}u(st_1, z) - t_1^{-1}u(st_2, z)) \eta'_{\varepsilon}(s-1) ds \right\vert\\ &\quad\quad+ \int_{\mathbb{R}}\lvert t_1^{-1} - t_2^{-1}\rvert\lvert u(st_2, z)\rvert \lvert \eta'_{\varepsilon}(s-1)\rvert ds \\
    &\quad\leq t_1^{-1} 2\kappa \lvert t_1-t_2\rvert \int_{\mathbb{R}}\lvert \eta'_{\varepsilon}(s-1)\rvert ds + \frac{\lvert t_1-t_2\rvert}{t_1t_2}\lVert u\rVert_{L^{\infty}(\Omega_T)}\int_{\mathbb{R}}\lvert \eta'_{\varepsilon}(s-1)\rvert ds.
\end{align*}
Here we used (\ref{kappa s}) for the last inequality. This implies that for each $\varepsilon > 0$, $\partial_t\Phi^{\varepsilon}$ is locally uniformly Lipschitz in $(0, S)$. One can also check that $\Phi^{\varepsilon}$ is locally uniformly semi-concave in $(0, S)$.\\    
    \indent Since $W^s$ satisfies (\ref{property of W^s}), we have
    \begin{align*}
        (dd^cW^s(t, \cdot))^n &\geq e^{\partial_tW^s(t, \cdot)+F(t, \cdot, W^s)}d\mu \\
        &\geq e^{\partial_tW^s(t, \cdot)-\partial_tv(t, \cdot)+F(t, \cdot, W^s)-F(t, \cdot, v)}(dd^cv(t, \cdot))^n
    \end{align*}
    for a.e. $t \in (0, S)$. Fix such a $t$.
    It follows from Lemma \ref{Lemma 3.1} that
    \begin{equation}\label{3.10}
        \begin{aligned}
            &(dd^cW^s(t, \cdot)) \wedge (dd^cv(t, \cdot))^{n-1}\\ &\quad\geq \exp \left(\frac{\partial_tW^s(t, \cdot)-\partial_tv(t, \cdot)+F(t, \cdot ,W^s)-F(t, \cdot, v)}{n}\right)(dd^cv(t, \cdot))^n.
        \end{aligned}
    \end{equation}
    We show that
    \begin{equation}\label{3.11}
        \begin{aligned}
            &(dd^c\Phi^{\varepsilon}(t, \cdot)) \wedge (dd^cv(t, \cdot))^{n-1}\\ &\geq \exp\left(\frac{\partial_t\Phi^{\varepsilon}(t, \cdot)-\partial_tv(t, \cdot)+F(t, \cdot, \Phi^{\varepsilon})-F(t, \cdot, v)-O(\varepsilon)}{n}\right) (dd^cv(t, \cdot))^n.
        \end{aligned}
    \end{equation}
    Indeed, let $\chi(z)$ be a positive smooth test function on $\Omega$. Then we have
    \begin{align*}
        &\int_{\Omega}\chi(z) (dd^c\Phi^{\varepsilon}(t, \cdot)) \wedge (dd^cv(t, \cdot))^{n-1}\\ &\quad= \int_{\Omega}\Phi^{\varepsilon}(t, \cdot)(dd^c\chi) \wedge (dd^cv(t, \cdot))^{n-1} \\&\quad=\int_{\Omega}\left(\int_{\mathbb{R}}W^s(t, \cdot)\eta_{\varepsilon}(s-1)ds\right)(dd^c\chi)\wedge (dd^cv(t, \cdot))^{n-1}. 
    \end{align*}
    Let us denote
    \begin{align*}
        I(s, z) = \partial_tW^s(t, z)-\partial_tv(t, z)+F(t, z, W^s)-F(t, z, v).
    \end{align*}
    By the Fubini's theorem, the last term satisfies
    \begin{align*}
        &\int_{\mathbb{R}}\eta_{\varepsilon}(s-1)\left(\int_{\Omega}W^s(t, \cdot)(dd^c\chi)\wedge (dd^cv(t, \cdot))^{n-1}\right)ds \\
        &\quad = \int_{\mathbb{R}}\eta_{\varepsilon}(s-1)\left(\int_{\Omega}\chi(z) (dd^cW^s(t, \cdot)) \wedge (dd^cv(t, \cdot))^{n-1}\right)ds\\
        &\quad \quad \geq \int_{\mathbb{R}}\eta_{\varepsilon}(s-1)\left(\int_{\Omega}\chi(z) \exp\left(\frac{I(s, z)}{n}\right)(dd^cv(t, \cdot))^n\right) ds.
    \end{align*}
    For the inequality, we used (\ref{3.10}). By the Fubini's theorem and Jensen's inequality, the last term satisfies
    \begin{equation}\label{3.12}
        \begin{aligned}
            &\int_{\Omega}\chi(z) \left(\int_{\mathbb{R}}\eta_{\varepsilon}(s-1)\exp\left(\frac{I(s, z)}{n}\right)ds \right) (dd^cv(t, \cdot))^n \\
        &\quad \geq \int_{\Omega}\chi(z) \exp\left(\frac{1}{n}\int_{\mathbb{R}}I(s, z)\eta_{\varepsilon}(s-1)ds\right) (dd^cv(t, \cdot))^n,
        \end{aligned}
    \end{equation}
    By following the last part of the proof of \cite[Theorem 6.6]{GLZ21-2} and using the assumption (\ref{lipschitzness near 0}), one can show that
    \begin{align*}
        \int_{\mathbb{R}}F(t, \cdot, W^s)\eta_{\varepsilon}(s-1)ds \geq F(t, \cdot, \Phi^{\varepsilon})-O(\varepsilon),
    \end{align*}
    which implies that
    \begin{align*}
        &\int_{\mathbb{R}}I(s, z)\eta_{\varepsilon}(s-1)ds \\& \quad= \int_{\mathbb{R}}\partial_t(W^s(t, z)-v(t ,z))\eta_{\varepsilon}(s-1)ds + \int_{\mathbb{R}}(F(t, z, W^s)-F(t, z, v))\eta_{\varepsilon}(s-1)ds \\
        &\quad \quad \geq \partial_t(\Phi^{\varepsilon}(t, z) - v(t, z)) + F(t, z, \Phi^{\varepsilon})-F(t, z, v)-O(\varepsilon).
    \end{align*}
    Therefore the last term of (\ref{3.12}) is larger than
    \begin{align*}
        \int_{\Omega}\chi(z) \exp\left(\frac{\partial_t(\Phi^{\varepsilon}(t, z)-v(t, z))+F(t, z, \Phi^{\varepsilon})-F(t, z, v)-O(\varepsilon)}{n}\right) (dd^cv(t, \cdot))^n.
    \end{align*}
    Since $\chi$ is chosen arbitrarily, we get (\ref{3.11}). \\
    \indent It remains to show the assumption (b) in Lemma \ref{Lemma 3.3}. To show this, we prove that 
    \begin{enumerate}
        \item [(i)] there exists a function $\eta(t) \in C([0, S))$ such that $\eta(0) = 0$ and
        \begin{align*}
            v(t, z) \geq v_0(z) - \eta(t)
        \end{align*}
        for $dt \wedge d\mu$-a.e. $(t, z) \in \Omega_T$,
        \item [(ii)] for each open $U \Subset \Omega$ and $\varepsilon > 0$, there exists $A = A(\varepsilon) > 0$ such that
        \begin{align*}
            \int_U (\Phi^{\varepsilon}(t, \cdot) - O(\varepsilon)t)d\mu \leq \int_U h_1(0, \cdot)d\mu+At
        \end{align*}
    \end{enumerate}
    and apply Lemma \ref{Lemma 3.6}.
    Note that (i) follows from Lemma \ref{Lemma 3.7}. To prove (ii), it suffices to show that for each open $U \Subset \Omega$, there exists $A_0 > 0$ such that
    \begin{align}\label{3.13}
        \int_{U}W^s(t, \cdot)d\mu \leq \int_U h_1(0, \cdot)d\mu + A_0t.
    \end{align}
    If so, by integrating each side with respect to $\eta_{\varepsilon}(s-1)ds$, we get
    \begin{align*}
        \int_U \Phi^{\varepsilon}(t, \cdot)d\mu \leq \int_{\mathbb{R}}\left(\int_U h_1(0, \cdot)d\mu + A_0t\right)\eta_{\varepsilon}(s-1)ds = \int_U h_1(0, \cdot)d\mu + A_0t,
    \end{align*}
    which implies (ii).
    The proof of (\ref{3.13}) is almost same with the proof of \cite[Lemma 3.10]{GLZ21-2}. Indeed, fix an open subset $U \Subset \Omega$ and let $\chi(z)$ be a positive test function on $\Omega$. Recall that $W^s$ satisfies
    \begin{align*}
        dt \wedge (dd^cW^s)^n \geq e^{\partial_tW^s+F(t, z, W^s)}dt \wedge d\mu.
    \end{align*}
    It follows from the elliptic Chern-Levine-Nirenberg inequality that there exists a constant $C > 0$ satisfying
    \begin{align*}
        \int_{U}\chi(z) e^{\partial_tW^s(t, \cdot)+F(t, \cdot, W^s)}d\mu
        \leq \int_{U}\chi(z) (dd^cW^{s}(t, \cdot))^n 
        \leq C
    \end{align*}
    for a.e. $t \in (0, T)$. By applying Jensen's inequality, there exists $A_0 > 0$ satisfying
    \begin{align*}
        \int_U \chi(z) \partial_tW^{s}(t, \cdot)d\mu \leq A_0,
    \end{align*}
    which implies that
    \begin{align*}
        \int_U W^s(t, \cdot)d\mu \leq \int_U h_1(0, \cdot)d\mu + A_0t.
    \end{align*}
    Hence it follows from Lemma \ref{Lemma 3.3} and Lemma \ref{Lemma 3.6} that $\Phi^{\varepsilon}-O(\varepsilon)t \leq v$ in $\Omega_S$. Since $\Phi^{\varepsilon} \rightarrow u$ pointwisely as $\varepsilon \rightarrow 0$, we get the result. \\
    \indent Finally, we remove our assumption (\ref{lipschitzness near 0}). Let $\kappa_F > 0$ be a constant satisfying
    \begin{align*}
        \lvert F(t_1, z, u) - F(t_2, z, u)\rvert \leq \kappa_F\lvert t_1-t_2\rvert
    \end{align*}
    for any $t_1, t_2 \in [0, S]$. Fix $0 < \varepsilon_0 < S-T$ and let $U^{\varepsilon_0}(t, z) := u(t+\varepsilon_0, z) - \kappa_F\varepsilon_0 t$. Then $U^{\varepsilon_0}$ satisfies
    \begin{align*}
        dt \wedge (dd^cU^{\varepsilon_0}(t, z))^n &= dt \wedge (dd^cu(t+\varepsilon_0, z))^n \\ &\quad \geq e^{\partial_tu(t+\varepsilon_0, z)+F(t+\varepsilon_0, z, u(t+\varepsilon_0, z))}dt \wedge d\mu \\
        &\quad \geq e^{\partial_tU^{\varepsilon_0}+\kappa_F\varepsilon_0 + F(t+\varepsilon_0, z, U^{\varepsilon_0})}dt \wedge d\mu \\
        &\quad \geq e^{\partial_tU^{\varepsilon_0}+F(t, z, U^{\varepsilon_0})}dt \wedge d\mu.
    \end{align*}
    For the second inequality, we used the assumption that $F$ is increasing in the third variable. Moreover, $U^{\varepsilon_0}$ is locally uniformly semi-concave in $(0, S)$ with the Cauchy-Dirichlet boundary data $h_1^{\varepsilon_0} := h_1-\kappa_F\varepsilon_0t$. Hence one may replace $u$ by $U^{\varepsilon_0}$ to conclude that $U^{\varepsilon_0} \leq v$ in $\Omega_S$. Since $U^{\varepsilon_0} \rightarrow u$ pointwisely as $\varepsilon_0 \rightarrow 0$ we get the result.
    \end{proof}
From this theorem, we get the following uniqueness result.
\begin{corollary}\label{Corollary 3.9}
    Let $u_0 \in PSH(\Omega) \cap L^{\infty}(\Omega)$ and $h \in C([0, T) \times \partial \Omega)$ satisfy
    \begin{itemize}
        \item $\lim_{z \rightarrow \partial\Omega}u_0(z) = h(0, \cdot)$,
        \item $h$ satisfies (\ref{additional assumption}) and (\ref{additional assumption 2}).
    \end{itemize}
    Then there exists a unique $u \in \mathcal{P}(\Omega_T) \cap L^{\infty}(\Omega_T)$ which is locally uniformly semi-concave in $(0, T)$ and satisfy
    \begin{align*}
        \begin{cases}
            &dt \wedge (dd^cu)^n = e^{\partial_tu+F(t, z, u)}dt \wedge d\mu \text{ in } \Omega_T, \\
            &\lim_{(t, z) \rightarrow (\tau, \zeta)}u(t, z) = h(\tau, \zeta) \text{ for all } (\tau, \zeta) \in [0, T)\times \partial \Omega, \\
            &\lim_{t \rightarrow 0+}u(t, \cdot) = h(0, \cdot) \text{ in } L^1(d\mu).
        \end{cases}
    \end{align*}
\end{corollary}
\begin{proof}
    The existence of the solution comes from \cite[Theorem 4.5]{Ka25}. By Theorem \ref{theorem 3.8}, we get the uniqueness of such a solution.
\end{proof}
\begin{remark}\label{infinite time case}
    Even though \cite[Theorem 4.5]{Ka25} only considered the case when $T < +\infty$, Corollary \ref{Corollary 3.9} covers the case when $T = +\infty$. Indeed, due to the uniqueness, one can obtain the solution by gluing each solutions defined on $\Omega_{T'}$ for each $T' < T$.
\end{remark}
Next theorem is another application of Theorem \ref{theorem 3.8}, which explains the long-term behavior of the solution. The motivation of the proof comes from \cite[Theorem 5.3]{GLZ20}.
\begin{theorem}
    Let us denote $T = +\infty$. Let $u \in \mathcal{P}(\Omega_T) \cap L^{\infty}(\Omega_T)$ be the unique function which is locally uniformly semi-concave in $(0, T)$ and satisfy
    \begin{align*}
        \begin{cases}
            & dt \wedge (dd^cu)^n = e^{\partial_tu+F(z, u)}dt \wedge d\mu \text{ in } \Omega_T, \\
            &\lim_{(t, z) \rightarrow (\tau, \zeta)}u(t, z) = 0 \text{ for all } (\tau, \zeta) \in [0, T) \times \partial \Omega, \\
            &\lim_{t \rightarrow 0+}u(t, \cdot) = \psi(\cdot) \text{ in } L^1(d\mu),
        \end{cases}
    \end{align*}
    where $\psi$ is given in (\ref{MA2}). If there exists a constant $L_F > 0$ such that 
    \begin{align}\label{L F}
        \lvert F(z, r_1) - F(z, r_2)\rvert \geq L_F\lvert r_1-r_2\rvert
    \end{align}
    for all $z \in \Omega$ and $r_1, r_2 \in J$, then $u(t, \cdot) \rightarrow \psi_{\infty}$ uniformly as $t \rightarrow +\infty$ for $\psi_{\infty} \in PSH(\Omega) \cap L^{\infty}(\Omega)$ satisfying
    \begin{align}\label{3.15}
        \begin{cases}
            &(dd^c\psi_{\infty})^n = e^{F(z, \psi_{\infty})}d\mu \text{ in } \Omega,\\
            &\lim_{z \rightarrow \partial \Omega}\psi_{\infty}(z) = 0.
        \end{cases}
    \end{align}
\end{theorem}
\begin{remark}
    It follows from \cite[Section 5]{Ka25} that $u$ actually satisfies $\lim_{t \rightarrow 0+}u(t, z) = \psi(z)$ for all $z \in \Omega$, since $(\mu, \psi)$ is admissible (see \cite[Definition 5.2]{Ka25}). 
\end{remark}
\begin{proof}
    It follows from \cite[Theorem 1.1]{Ko00} that there exists $\psi_{\infty}$ satisfying (\ref{3.15}). It remains to show that $u(t, \cdot) \rightarrow \psi_{\infty}$ as $t \rightarrow +\infty$. \\
    \indent Let us define
    \begin{align*}
        \Phi(t) := \frac{n}{L_F}(e^{L_Ft}-1)\log(e^{L_Ft}-1)-nte^{L_Ft}.
    \end{align*}
    One can check that $\Phi$ satisfies
    \begin{enumerate}
        \item [(i)] $\Phi(t) \leq 0$ for all $t \in [0, T)$,
        \item [(ii)] $\Phi'(t) = ne^{L_Ft}\log(1-e^{-L_Ft})$,
        \item [(iii)] $\Phi(t)e^{-L_Ft}$ is bounded and $\Phi(t)e^{-L_Ft} \rightarrow 0$ as $t \rightarrow \infty$.
    \end{enumerate}
    Next, by the local uniform Lipschitzness of $F$, there exists $\kappa_F > 0$ such that
    \begin{align}\label{kappa F}
        \lvert F(z, r_1) - F(z, r_2) \rvert \leq \kappa_F \lvert r_1-r_2\rvert 
    \end{align}
    for all $r_1, r_2 \in (-M, M)$ where
    \begin{align*}
        M := \lVert \psi\rVert_{L^{\infty}(\Omega)}+\lVert \psi_{\infty}\rVert_{L^{\infty}(\Omega)} + \sup_{t \geq 0}\lvert \Phi(t)e^{-L_Ft}\rvert.
    \end{align*}
    Let us define
    \begin{align*}
        W(t, z) : = e^{-L_Ft}\psi(z) + (1-e^{-\kappa_Ft})\psi_{\infty}(z) + \Phi(t)e^{-L_Ft}.
    \end{align*}
    Since $\psi, \psi_{\infty} \leq 0$ on $\Omega$ and $\Phi \leq 0$ on $[0, T)$,
    \begin{align*}
        \partial_tW+F(z, W) &= -L_Fe^{-L_Ft}(\psi+\Phi)+ \kappa_Fe^{-\kappa_Ft}\psi_{\infty}+ \Phi'e^{-L_Ft}+F(z, W) \\
        &\quad \leq F(z, \psi_{\infty}+e^{-L_Ft}(\psi+\Phi))-L_Fe^{-L_Ft}(\psi+\Phi)+\Phi'e^{-L_Ft} \\
        &\quad \leq F(z, \psi_{\infty})+n\log(1-e^{-L_Ft}).
    \end{align*}
    Here we used (\ref{kappa F}) for the first inequality and (\ref{L F}) for the second inequality.
    Therefore $W \in \mathcal{P}(\Omega_T) \cap L^{\infty}(\Omega_T)$ satisfies
    \begin{align*}
        dt \wedge (dd^cW)^n &\geq (1-e^{-\kappa_Ft})^ndt \wedge (dd^c\psi_{\infty})^n \\
        &\geq e^{n\log(1-e^{-L_Ft})+F(z, \psi_{\infty})}dt \wedge d\mu \\
        &\geq e^{\partial_tW+F(z, W)}dt \wedge d\mu.
    \end{align*}
    Moreover, $W$ is locally uniformly semi-concave in $(0, T)$ and satisfies
    \begin{align*}
        \begin{cases}
            &\lim_{(t, z) \rightarrow (\tau, \zeta)}W(t, z) = \Phi(\tau)e^{-L_F\tau} \leq 0 \text{ for all } (\tau, \zeta) \in [0, T) \times \partial \Omega, \\
            &\lim_{t \rightarrow 0+}W(t, z) = \psi(z) \text{ for all } z\in \Omega.
        \end{cases}
    \end{align*}
    Hence it follows from Theorem \ref{theorem 3.8} that
    \begin{align*}
        W(t, z) \leq u(t, z),
    \end{align*}
    which implies that
    \begin{align}\label{3.17}
        \psi_{\infty}(z) - u(t, z) \leq e^{-\kappa_Ft}\psi_{\infty}(z)-e^{-L_Ft}(\psi(z)+\Phi(t)).
    \end{align}
    \indent Next, let us define
    \begin{align*}
        V(t, z) := (1-e^{-ct})\psi_{\infty}+e^{-ct}
    \end{align*}
    where 
    \begin{align*}
        c := \frac{n}{1+\lVert \psi_{\infty}\rVert_{L^{\infty}(\Omega)}} > 0.
    \end{align*}
    Note that $V \in \mathcal{P}(\Omega_T) \cap L^{\infty}(\Omega_T)$ is locally uniformly semi-concave in $(0, T)$ and satisfies
    \begin{align*}
        \begin{cases}
            &\lim_{(t, z) \rightarrow (\tau, \zeta)}V(t, z) = e^{-c\tau} > 0 \text{ for all } (\tau, \zeta) \in [0, T) \times \partial \Omega, \\
            &\lim_{t \rightarrow 0+}V(t, z) = 1 > \psi(z) \text{ for all } z \in \Omega.
        \end{cases}
    \end{align*}
    Moreover, since $F = F(z ,r)$ is increasing in $r$, 
    \begin{align*}
        \partial_tV+F(z, V) 
        & \geq ce^{-ct}(\psi_{\infty}-1) + F(z, \psi_{\infty}) \\
        &\geq -ne^{-ct}+F(z, \psi_{\infty}) \\
        &\geq n\log(1-e^{-ct})+F(z, \psi_{\infty}).
    \end{align*}
    Here we used the inequality that $V \geq \psi_{\infty}$ for the first inequality. Therefore $V$ satisfies
    \begin{align*}
        dt \wedge (dd^cV)^n &= (1-e^{-ct})^n dt \wedge (dd^c\psi_{\infty})^n \\
        &= e^{n\log(1-e^{-ct})+F(z, \psi_{\infty})}dt \wedge d\mu \\
        &\quad \leq e^{\partial_tV+F(z, V)}dt \wedge d\mu.
    \end{align*}
    It follows from Theorem \ref{theorem 3.8} that
    \begin{align*}
        V(t, z) \geq u(t, z),
    \end{align*}
    which implies that
    \begin{align}\label{3.18}
        u(t, z) - \psi_{\infty}(z) \leq e^{-ct}(1-\psi_{\infty}).
    \end{align}
    \indent Finally, by (\ref{3.17}) and (\ref{3.18}), $u(t, \cdot) \rightarrow \psi_{\infty}$ uniformly as $t \rightarrow \infty$.
\end{proof}
\bibliographystyle{abbrv}
\bibliography{ref.bib}
\end{document}